\documentclass{article}
\usepackage{amsfonts}
\usepackage{amsmath}
\usepackage{harvard}

\newtheorem{theorem}{Theorem}

\newtheorem{lemma}[theorem]{Lemma}

\newenvironment{proof}[1][Proof]{\noindent\textbf{#1.} }{\ \rule{0.5em}{0.5em}}

\input{tcilatex}

\begin{document}

\title{On the Determinants and Inverses of Circulant Matrices with a General
Number Sequence}
\author{Durmu\c{s} Bozkurt\thanks{%
e-mail: dbozkurt@selcuk.edu.tr} \and Department of Mathematics, Science
Faculty, Sel\c{c}uk University, \and 42075 Kampus, Konya, Turkey}
\maketitle

\begin{abstract}
The generalized sequence of numbers is defined by $W_{n}=pW_{n-1}+qW_{n-2}$
with initial conditions $W_{0}=a$ and $W_{1}=b$ for $a,b,p,q\in \mathbb{Z}$
and $n\geq 2,$ respectively. Let $\mathbb{W}_{n}=circ(W_{1},W_{2},\ldots
,W_{n}).$\ The \ aim of this paper is to establish some useful formulas for
the determinants and inverses of $\mathbb{W}_{n}$ using the nice properties
of the number sequences. Matrix decompositions are derived for $\mathbb{W}%
_{n}$\ in order to obtain the results.
\end{abstract}

\section{Introduction}

\bigskip The $n\times n$ circulant matrix $C_{n}=circ(c_{0},c_{1},\ldots
,c_{n-1}),$\ associated with the numbers $c_{0},c_{1},\ldots ,c_{n-1},$\ is
defined by%
\begin{equation}
C_{n}:=\left[ 
\begin{array}{ccccc}
c_{0} & c_{1} & \ldots & c_{n-2} & c_{n-1} \\ 
c_{n-1} & c_{0} & \ldots & c_{n-3} & c_{n-2} \\ 
\vdots & \vdots & \ddots & \vdots & \vdots \\ 
c_{2} & c_{3} & \ldots & c_{0} & c_{1} \\ 
c_{1} & c_{2} & \ldots & c_{n-1} & c_{0}%
\end{array}%
\right] .  \label{1}
\end{equation}

\bigskip Circulant matrices have a {wide} range of {applications, for
examples} in signal processing, coding theory, image processing, digital
image disposal, self-regress {\ design} and so on. {Numerical} solutions of
the certain types of elliptic and parabolic partial differential equations
with periodic boundary conditions often involve linear systems {associated
with circulant matrices} [9-11].

\bigskip The eigenvalues and eigenvectors of $C_{n}$ are well-known [4,14]:%
\begin{equation*}
\lambda _{j}=\dsum\limits_{k=0}^{n-1}c_{k}\omega ^{jk},\ \ \ \ \
j=0,1,\ldots ,n-1,
\end{equation*}%
where $\omega :=\exp (\frac{2\pi i}{n})$ and $i:=\sqrt{-1}$ and the
corresponding eigenvectors%
\begin{equation*}
x_{j}=(1,\omega ^{j},\omega ^{2j},\ldots ,\omega ^{(n-1)j}),\ \ \
j=0,1,\ldots ,n-1.
\end{equation*}

Thus we have the determinants and inverses of nonsingular circulant matrices
[1,3,4,14]:%
\begin{equation*}
\det (C_{n})=\dprod\limits_{j=0}^{n-1}(\dsum\limits_{k=0}^{n-1}c_{k}\omega
^{jk})
\end{equation*}%
and%
\begin{equation*}
C_{n}^{-1}=circ(a_{0},a_{1},\ldots ,a_{n-1})
\end{equation*}%
where $a_{k}=\frac{1}{n}\tsum\nolimits_{j=0}^{n-1}\lambda _{j}\omega ^{-jk},$
and $k=0,1,\ldots ,n-1)~[4].$ When $n$ is getting large, the above the
determinant and inverse formulas are not very handy to use. If there is some
structure among $c_{0},c_{1},\ldots ,c_{n-1},$\ we may be able to get more
explicit forms of the eigenvalues, determinants and inverses of $C_{n}.$
Recently, studies on the circulant matrices involving interesting number
sequences appeared. In [1] the determinants and inverses of the circulant
matrices $A_{n}=circ(F_{1},F_{2},\ldots ,F_{n})$ and $%
B_{n}=circ(L_{1},L_{2},\ldots ,L_{n})$\ are derived where $F_{n}$ and $L_{n}$%
\ are $n$th Fibonacci and Lucas numbers, respectively. In [2] the $r$%
-circulant matrix is defined and its norms was computed. The norms of
Toeplitz matrices involving Fibonacci and Lucas numbers are obtained [5].
Miladinovic and Stanimirovic [6] gave an explicit formula of the
Moore-Penrose inverse of singular generalized Fibonacci matrix. Lee and et
al. found the factorizations and eigenvalues of Fibonacci and symmetric
Fibonacci matrices [7].

\bigskip The generalized sequence $\{W_{n}(a,b;p,q)\}$ or $\{W_{n}\}$\ of
numbers is defined by $W_{n}=pW_{n-1}-qW_{n-2}$ with initial conditions $%
W_{0}=a$ and $W_{1}=b$ for $a,b,p,q\in \mathbb{Z}$ and $n\geq 2,$
respectively [16, pp. 161]. Let $\alpha $ and $\beta $ be the roots of $%
x^{2}-px+q=0.$ Then the Binet formula of the sequence $\{W_{n}\}$\ is%
\begin{equation*}
W_{n}=\frac{A\alpha ^{n}+B\beta ^{n}}{\alpha -\beta }\ \text{[16, pp. 161]}
\end{equation*}%
where $A=b-a\beta $ and $B=a\alpha -b,\alpha +\beta =p,\alpha \beta =q$ and $%
\alpha -\beta =$ $\sqrt{p^{2}-4q}.$ Let $p=q=1.$ If $a=0,b=1$ and $a=2,b=1,$
then $G_{n}$ are $F_{n}$ $n$th Fibonacci and $L_{n}$ $n$th Lucas numbers,
respectively. While $p=2,q=1,$ if $a=0,b=1$ and $a=b=2,$ then $G_{n}$ are $%
P_{n}$ $n$th Pell and $Q_{n}$ $n$th Pell-Lucas numbers, respectively.

Let $\mathbb{W}=circ(W_{1},W_{2},\ldots ,W_{n}).$\ The \ aim of this paper
is to establish some useful formulas for the determinants and inverses of $%
\mathbb{W}$ using the nice properties of the number sequences. Matrix
decompositions are derived for $\mathbb{W}$\ in order to obtain the results.

\section{Determinants of circulant matrices with the\protect\linebreak %
number sequence}

Recall that $\mathbb{W}_{n}=circ(W_{1},W_{2},\ldots ,W_{n}),$ i.e. where $%
W_{k}$ is $k$th element of sequence $\{W_{n}\}$, with the recurence
relations $W_{k}=pW_{k-1}+qW_{k-2},$ the initial conditions $W_{0}=a,$ $%
W_{1}=b$\ $(k\geq 2)\ $when $q$ is positive real number. Let $\alpha $\ and $%
\beta $\ be the roots of $x^{2}-px+q=0\ $while $p^{2}-4q\neq 0.$ Using the
Binet formula [16, pp. 161] for the sequence $\{W_{n}\},$\ one has%
\begin{equation}
W_{n}=\frac{A\alpha ^{n}+B\beta ^{n}}{\alpha -\beta }\ \text{[16, pp. 161]}
\label{2}
\end{equation}%
\ 

\begin{theorem}
Let \bigskip $n\geq 3.$ Then%
\begin{eqnarray}
\det (\mathbb{W}_{n})
&=&(b^{2}-W_{2}W_{n})(b-W_{n+1})^{n-2}+\dsum%
\limits_{k=2}^{n-1}[(bW_{k+1}-W_{2}W_{k})\times  \notag \\
&&\times (b-W_{n+1})^{k-2}(qW_{n}-qa)^{n-k}]  \label{3}
\end{eqnarray}%
where $W_{k}$ is $k$th the element of the sequence $\{W_{n}\}$.
\end{theorem}

\begin{proof}
Obviously, $\det (\mathbb{W}_{3})=b^{3}+W_{2}^{3}+W_{3}^{3}-3bW_{2}W_{3}.$
It satisfies (\ref{3}). For $n>3,$ we select the matrices $K$ and $L$ so
that when we multiply $\mathbb{W}_{n}$\ with $K$ on the left and $L$ on the
right we obtain a special Hessenberg matrix that have nonzero entries only
on first two rows, main diagonal and subdiagonal: \ 
\begin{equation}
K:=\left[ 
\begin{array}{ccccccc}
1 & 0 & 0 & 0 & \ldots  & 0 & 0 \\ 
-\frac{W_{2}}{W_{1}} & 0 & 0 & 0 & \ldots  & 0 & 1 \\ 
-q & 0 & 0 & 0 & \ldots  & 1 & -p \\ 
0 & 0 & 0 & 0 & \ldots  & -p & -q \\ 
\vdots  & \vdots  & \vdots  & \vdots  & \ddots  & \vdots  & \vdots  \\ 
0 & 0 & 1 & -p & \ldots  & 0 & 0 \\ 
0 & 1 & -p & -q & \ldots  & 0 & 0%
\end{array}%
\right]   \label{4}
\end{equation}%
and%
\begin{equation*}
L:=\left[ 
\begin{array}{cccccc}
1 & 0 & 0 & \ldots  & 0 & 0 \\ 
0 & \left( \frac{q(W_{n}-W_{0})}{W_{1}-W_{n+1}}\right) ^{n-2} & 0 & \ldots 
& 0 & 1 \\ 
0 & \left( \frac{q(W_{n}-W_{0})}{W_{1}-W_{n+1}}\right) ^{n-3} & 0 & \ldots 
& 1 & 0 \\ 
0 & \left( \frac{q(W_{n}-W_{0})}{W_{1}-W_{n+1}}\right) ^{n-4} & 0 & \ldots 
& 0 & 0 \\ 
\vdots  & \vdots  & \vdots  &  & \vdots  & \vdots  \\ 
0 & \left( \frac{q(W_{n}-W_{0})}{W_{1}-W_{n+1}}\right)  & 1 & \ldots  & 0 & 0
\\ 
0 & 1 & 0 & \ldots  & 0 & 0%
\end{array}%
\right] .
\end{equation*}%
Notice that we obtain the following equivalence:%
\begin{eqnarray*}
M &=&K\mathbb{W}_{n}L \\
&=&\left[ 
\begin{array}{ccccc}
W_{1} & g_{n}^{\prime } & W_{n-1} & W_{n-2} & W_{n-3} \\ 
& g_{n} & W_{n}-\frac{W_{2}W_{n-1}}{W_{1}} & W_{n-1}-\frac{W_{2}W_{n-2}}{%
W_{1}} & W_{n-2}-\frac{W_{2}W_{n-3}}{W_{1}} \\ 
&  & W_{1}-W_{n+1} &  &  \\ 
&  & q(W_{0}-W_{n}) & W_{1}-W_{n+1} &  \\ 
&  &  & q(W_{0}-W_{n}) & W_{1}-W_{n+1} \\ 
&  &  &  & q(W_{0}-W_{n}) \\ 
& 0 &  &  &  \\ 
&  &  &  & 
\end{array}%
\right.  \\
&&\ \ \ \ \ \ \ \ \ \ \ \ \ \ \ \ \ \ \ \ \ \ \ \ \ \ \ \ \ \ \ \ \ \ \ \ \
\ \ \ \ \ \ \ \left. 
\begin{array}{ccc}
\ldots  & W_{3} & W_{2} \\ 
\ldots  & W_{4}-\frac{W_{2}W_{3}}{W_{1}} & W_{3}-\frac{W_{2}^{2}}{W_{1}} \\ 
&  &  \\ 
&  & 0 \\ 
&  &  \\ 
\ddots  &  &  \\ 
\ddots  & W_{1}-W_{n+1} &  \\ 
& q(W_{0}-W_{n}) & W_{1}-W_{n+1}%
\end{array}%
\right] 
\end{eqnarray*}%
and $M$ is Hessenberg matrix, where 
\begin{equation*}
g_{n}^{\prime }:=\dsum\limits_{k=2}^{n}W_{k}\left( \frac{q(W_{n}-W_{0})}{%
W_{1}-W_{n+1}}\right) ^{n-k}
\end{equation*}%
and%
\begin{equation*}
g_{n}:=W_{_{1}}-\frac{W_{2}W_{n}}{W_{1}}+\dsum\limits_{k=2}^{n-1}\left(
\left( W_{k+1}-\frac{W_{2}W_{k}}{W_{1}}\right) \left( \frac{q(W_{n}-W_{0})}{%
W_{1}-W_{n+1}}\right) ^{n-k}\right) .
\end{equation*}%
Then we have%
\begin{equation*}
\det (M)=\det (K)\det (\mathbb{W}_{n})\det (L)=b(b-W_{n+1})^{n-2}g_{n}.
\end{equation*}%
Since%
\begin{equation*}
\det (K)=\det (L)=\left\{ 
\begin{array}{l}
\ \ 1,\ n\equiv 1\ or\ 2\ \func{mod}4 \\ 
-1,\ n\equiv 0\ or\ 3\ \func{mod}4%
\end{array}%
\right. 
\end{equation*}%
for all $n>3,$%
\begin{equation*}
\det (K)\det (L)=1
\end{equation*}%
and (\ref{3}) follows.
\end{proof}

\section{\protect\bigskip Inverses of $\mathbb{W}_{n}$}

Let $C_{k,n}$ be an $n\times n$ $k$-circulant matrix. Then%
\begin{equation*}
\Delta _{C_{k,n}}(\lambda )=\lambda ^{n-m}\dprod\limits_{j=0}^{r-1}(\lambda
^{n_{j}}-y_{j})
\end{equation*}%
is the characteristic polynomial of the matrix $C_{k,n}$ where%
\begin{equation*}
y_{j}=\dprod\limits_{s\in P_{j}}\lambda _{ty},\ \ \ j=0,1,\ldots ,r-1
\end{equation*}%
and%
\begin{equation*}
y=\frac{n}{m}\text{ (See [15], pp. 3)}.
\end{equation*}%
Since $GCD(1,n)=1$, then $n=m$ and the characterictic polynomial of the
matrix $C_{1,n}$ is%
\begin{equation*}
\Delta _{C_{1,n}}(\lambda )=\dprod\limits_{j=0}^{r-1}(\lambda
^{n_{j}}-y_{j}).
\end{equation*}%
Therefore, 0 is not the eigenvalue of $C_{1,n}.$ Then $\det (C_{1,n})\neq 0$
for $n\geq 3$.

We will use the well-known fact that the inverse of a nonsingular circulant
matrix is also circulant [14, p.84] [12, p.33] [4, p.90-91].

\begin{lemma}
Let $A=\bigskip (a_{ij})$ be an $(n-2)\times (n-2)$ matrix defined by%
\begin{equation*}
a_{ij}=\left\{ 
\begin{array}{l}
W_{1}-W_{n+1},\ \ \ \ i=j \\ 
q(W_{0}-W_{n}),\ i=j+1 \\ 
0,\ \ \ \ \ \ \ \ \ \ otherwise.%
\end{array}%
\right.
\end{equation*}%
Then $A^{-1}=(a_{ij}^{^{\prime }})$ is given by%
\begin{equation*}
a_{ij}^{\prime }=\left\{ 
\begin{array}{l}
\frac{(q(W_{0}-W_{n}))^{i-j}}{(W_{1}-W_{n+1})^{i-j+1}},i\geq j \\ 
0,\ \ \ \ \ \ \ \ \ \ otherwise.%
\end{array}%
\right.
\end{equation*}
\end{lemma}

\begin{proof}
Let $B:=(b_{ij})=BB^{-1}.$ Clearly, $b_{ij}=\tsum%
\nolimits_{k=1}^{n-2}a_{ik}a_{kj}^{^{\prime }}.$ When $i=j,$ we have%
\begin{equation*}
b_{ii}=(W_{1}-W_{n+1}).\frac{1}{W_{1}-W_{n+1}}=1.
\end{equation*}%
If $i>j,$ then%
\begin{eqnarray*}
b_{ij} &=&\tsum\nolimits_{k=1}^{n-2}a_{ik}a_{kj}^{^{\prime
}}=a_{i,i-1}a_{i-1,j}^{^{\prime }}+a_{ii}a_{ij}^{^{\prime }} \\
&=&q(W_{0}-W_{n}).\frac{(qW_{0}-qW_{n})^{i-j-1}}{(W_{1}-W_{n+1})^{i-j}}%
+(W_{1}-W_{n+1})\frac{(qW_{0}-qW_{n})^{i-j}}{(W_{1}-W_{n+1})^{i-j+1}}=0;
\end{eqnarray*}%
similar for $i<j.$ Thus, $BB^{-1}=I_{n-2}$.
\end{proof}

\begin{theorem}
\bigskip Let the matrix $\mathbb{W}_{n}\mathbb{\ }$be $\mathbb{W}_{n}=$ $%
circ(W_{1},W_{2},\ldots ,W_{n})$ ($n\geq 3$). Then the inverse of the matrix 
$\mathbb{W}_{n}$ is%
\begin{equation*}
\mathbb{W}_{n}^{-1}=circ(w_{1},w_{2},\ldots ,w_{n})
\end{equation*}%
where%
\begin{eqnarray*}
w_{1} &=&\frac{1}{g_{n}}-\frac{1}{g_{n}(W_{1}-W_{n+1})}\left( p\left( W_{n}-%
\frac{W_{2}W_{n-1}}{W_{1}}\right) +\right. \\
&&\left. \dsum\limits_{k=1}^{n-2}q^{k}\left( W_{n-k}-\frac{W_{2}W_{n-k-1}}{%
W_{1}}\right) \left( \frac{W_{0}-W_{n}}{W_{1}-W_{n+1}}\right)
^{{}^{k-1}}\left( 1+p\frac{W_{0}-W_{n}}{W_{1}-W_{n+1}}\right) \right) . \\
w_{2} &=&-\frac{W_{2}}{g_{n}W_{1}}-\frac{1}{g_{n}(W_{1}-W_{n+1})}%
\dsum\limits_{k=1}^{n-2}\left( \frac{q(W_{0}-W_{n})}{W_{1}-W_{n+1}}\right)
^{k-1} \\
&&\times \left( W_{n-k+1}-\frac{W_{2}W_{n-k}}{W_{1}}\right) \\
w_{3} &=&\frac{W_{1}W_{3}-W_{2}^{2}}{g_{n}W_{1}(W_{1}-W_{n+1})} \\
w_{4} &=&\frac{1}{g_{n}(W_{1}-W_{n+1})}\left[ W_{4}-\frac{W_{2}W_{3}}{W_{1}}%
+\left( W_{3}-\frac{W_{2}^{2}}{W_{1}}\right) \left( \frac{W_{n+2}-W_{2}}{%
W_{1}-W_{n+1}}\right) \right] \\
w_{5} &=&\frac{q}{g_{n}(W_{1}-W_{n+1})}\left[ \left( W_{3}-\frac{W_{2}^{2}}{%
W_{1}}\right) \left( \frac{(W_{0}-W_{n})(W_{n+2}-W_{2})}{(W_{1}-W_{n+1})^{2}}%
-1\right) \right] \\
&&\vdots \\
w_{n} &=&\frac{1}{g_{n}(W_{1}-W_{n+1})}\left[ \left( W_{n}-\frac{W_{2}W_{n-1}%
}{W_{1}}\right) +\left( W_{n-1}-\frac{W_{2}W_{n-2}}{W_{1}}\right) \times
\right. \\
&&\times \left( \frac{W_{n+2}-W_{2}}{W_{1}-W_{n+1}}\right)
+\dsum\limits_{k=2}^{n-2}q^{k-1}\left( W_{n-k}-\frac{W_{2}W_{n-k-1}}{W_{1}}%
\right) \times \\
&&\times \left( \frac{W_{0}-W_{n}}{W_{1}-W_{n+1}}\right) ^{{}^{k-2}}\left( 
\frac{(W_{0}-W_{n})(W_{n+2}-W_{2})}{(W_{1}-W_{n+1})^{2}}-1\right)
\end{eqnarray*}%
for $g_{n}:=W_{_{1}}-\frac{W_{2}W_{n}}{W_{1}}+\dsum\limits_{k=2}^{n-1}\left(
\left( W_{k+1}-\frac{W_{2}W_{k}}{W_{1}}\right) \left( \frac{q(W_{n}-W_{0})}{%
W_{1}-W_{n+1}}\right) ^{n-k}\right) .$
\end{theorem}

\begin{proof}
Let%
\begin{eqnarray*}
U &:&=\left[ 
\begin{array}{ccccc}
1 & -\frac{g_{n}^{\prime }}{W_{1}} & u_{13} & u_{14} & u_{15} \\ 
& 1 & \frac{W_{n}}{g_{n}}-\frac{W_{2}W_{n-1}}{g_{n}W_{1}} & \frac{W_{n-1}}{%
g_{n}}-\frac{W_{2}W_{n-2}}{g_{n}W_{1}} & \frac{W_{n-2}}{g_{n}}-\frac{%
W_{2}W_{n-3}}{g_{n}W_{1}} \\ 
&  & 1 &  &  \\ 
&  & 0 & 1 &  \\ 
&  &  & 0 & 1 \\ 
&  &  &  & 0 \\ 
& 0 &  &  &  \\ 
&  &  &  & 
\end{array}%
\right. \\
&&\ \ \ \ \ \ \ \ \ \ \ \ \ \ \ \ \ \ \ \ \ \ \ \ \ \ \ \ \ \ \ \ \ \ \ \ \
\ \ \ \ \ \ \ \ \ \ \ \ \left. 
\begin{array}{ccc}
\ldots & u_{1,n-1} & u_{1n} \\ 
\ldots & \frac{W_{4}}{g_{n}}-\frac{W_{2}W_{3}}{g_{n}W_{1}} & \frac{W_{3}}{%
g_{n}}-\frac{W_{2}^{2}}{g_{n}W_{1}} \\ 
&  &  \\ 
&  & 0 \\ 
&  &  \\ 
\ddots &  &  \\ 
\ddots & 1 &  \\ 
& 0 & 1%
\end{array}%
\right]
\end{eqnarray*}%
where 
\begin{equation*}
u_{1j}=\frac{g_{n}^{\prime }}{g_{n}W_{1}}\left( \frac{W_{2}W_{n-j+2}}{W_{1}}%
-W_{n-j+3}\right) -\frac{W_{n-j+2}}{W_{1}},\ \ j=3,4,\ldots ,n
\end{equation*}%
and 
\begin{equation*}
g_{n}^{\prime }=\dsum\limits_{k=2}^{n}W_{k}\left( \frac{q(W_{n}-W_{0})}{%
W_{1}-W_{n+1}}\right) ^{n-k}
\end{equation*}
and 
\begin{equation*}
g_{n}=W_{_{1}}-\frac{W_{2}W_{n}}{W_{1}}+\dsum\limits_{k=2}^{n-1}\left(
\left( W_{k+1}-\frac{W_{2}W_{k}}{W_{1}}\right) \left( \frac{q(W_{n}-W_{0})}{%
W_{1}-W_{n+1}}\right) ^{n-k}\right) .
\end{equation*}
Let $H=diag(W_{1},g_{n}).$ Then we can write%
\begin{equation*}
K\mathbb{W}_{n}LU=H\oplus A
\end{equation*}%
where $H\oplus A$ is the direct sum of the matrices $H$ and $A$. Let $T=LU.$
Then we have%
\begin{equation*}
\mathbb{W}_{n}^{-1}=T(H^{-1}\oplus A^{-1})K.
\end{equation*}

Since the matrix $\mathbb{W}_{n}$ is circulant, its inverse is circulant
from Lemma 1.1 [1, p.9791]. Let%
\begin{equation*}
\mathbb{W}_{n}^{-1}:=circ(w_{1},w_{2},\ldots ,w_{n}).
\end{equation*}%
Since the last row of the matrix $T$ is 
\begin{eqnarray*}
&&\left( 0,1,\frac{W_{n}}{g_{n}}-\frac{W_{2}W_{n-1}}{g_{n}W_{1}},\frac{%
W_{n-1}}{g_{n}}-\frac{W_{2}W_{n-2}}{g_{n}W_{1}},\frac{W_{n-2}}{g_{n}}-\frac{%
W_{2}W_{n-3}}{g_{n}W_{1}},\right. \\
&&\ \ \ \ \ \ \ \ \ \ \ \ \ \ \ \ \ \ \ \ \ \ \ \ \ \ \ \ \ \ \ \ \ \ \ \ \
\ \ \ \ \ \ \left. \ldots ,\frac{W_{4}}{g_{n}}-\frac{W_{2}W_{3}}{g_{n}W_{1}},%
\frac{W_{3}}{g_{n}}-\frac{W_{2}^{2}}{g_{n}W_{1}}\right) ,
\end{eqnarray*}%
the last row components of the matrix $\mathbb{W}_{n}^{-1}$ are%
\begin{eqnarray*}
w_{2} &=&-\frac{W_{2}}{g_{n}W_{1}}-\frac{1}{g_{n}(W_{1}-W_{n+1})}%
\dsum\limits_{k=1}^{n-2}\left( \frac{q(W_{0}-W_{n})}{W_{1}-W_{n+1}}\right)
^{k-1} \\
&&\times \left( W_{n-k+1}-\frac{W_{2}W_{n-k}}{W_{1}}\right) \\
w_{3} &=&\frac{W_{1}W_{3}-W_{2}^{2}}{g_{n}W_{1}(W_{1}-W_{n+1})} \\
w_{4} &=&\frac{1}{g_{n}(W_{1}-W_{n+1})}\left[ W_{4}-\frac{W_{2}W_{3}}{W_{1}}%
+\left( W_{3}-\frac{W_{2}^{2}}{W_{1}}\right) \left( \frac{W_{n+2}-W_{2}}{%
W_{1}-W_{n+1}}\right) \right] \\
w_{5} &=&\frac{q}{g_{n}(W_{1}-W_{n+1})}\left[ \left( W_{3}-\frac{W_{2}^{2}}{%
W_{1}}\right) \left( \frac{(W_{0}-W_{n})(W_{n+2}-W_{2})}{(W_{1}-W_{n+1})^{2}}%
-1\right) \right] \\
&&\vdots \\
w_{n} &=&w_{n}=\frac{1}{g_{n}(W_{1}-W_{n+1})}\left[ \left( W_{n}-\frac{%
W_{2}W_{n-1}}{W_{1}}\right) +\left( W_{n-1}-\frac{W_{2}W_{n-2}}{W_{1}}%
\right) \times \right. \\
&&\times \left( \frac{W_{n+2}-W_{2}}{W_{1}-W_{n+1}}\right)
+\dsum\limits_{k=2}^{n-2}q^{k-1}\left( W_{n-k}-\frac{W_{2}W_{n-k-1}}{W_{1}}%
\right) \left( \frac{W_{0}-W_{n}}{W_{1}-W_{n+1}}\right) ^{{}^{k-2}} \\
&&\left. \times \left( \frac{(W_{0}-W_{n})(W_{n+2}-W_{2})}{%
(W_{1}-W_{n+1})^{2}}-1\right) \right] \\
w_{1} &=&\frac{1}{g_{n}}-\frac{1}{g_{n}(W_{1}-W_{n+1})}\left[ p\left( W_{n}-%
\frac{W_{2}W_{n-1}}{W_{1}}\right) +\right. \\
&&+\dsum\limits_{k=1}^{n-2}q^{k}\left( W_{n-k}-\frac{W_{2}W_{n-k-1}}{W_{1}}%
\right) \left( \frac{W_{0}-W_{n}}{W_{1}-W_{n+1}}\right) ^{{}^{k-1}}\times \\
&&\left. \times \left( 1+p\frac{W_{0}-W_{n}}{W_{1}-W_{n+1}}\right) \right]
\end{eqnarray*}%
where $g_{n}=W_{_{1}}-\frac{W_{2}W_{n}}{W_{1}}+\dsum\limits_{k=2}^{n-1}%
\left( \left( W_{k+1}-\frac{W_{2}W_{k}}{W_{1}}\right) \left( \frac{%
q(W_{n}-W_{0})}{W_{1}-W_{n+1}}\right) ^{n-k}\right) .$ Since $\mathbb{W}%
_{n}^{-1}$ is circulant matrix, the proof is completed.
\end{proof}

\end{document}